\documentclass[12pt]{amsart}
\usepackage{amsfonts, amsmath, amsthm}
\usepackage{graphicx}

\newtheorem{pro}{Proposition}
\newtheorem{thm}[pro]{Theorem}
\newtheorem{lem}[pro]{Lemma}

\newtheorem{question}[pro]{Question}

\theoremstyle{definition}
\newtheorem{dfn}[pro]{Definition}
\newtheorem{dfns}[pro]{Definitions}
\newtheorem{rmkk}[pro]{Remark}

\theoremstyle{remark}

\newcommand{\s}{\Sigma}

\newcommand{\ie}{{\it i.e.}}

\newcommand{\del}{\partial}

\title[reducibility of Heegaard splittings]
{Manifolds admitting both strongly irreducible and weakly reducible
  minimal genus {H}eegaard splittings}

\date{\today} \address{Department of Mathematics, Nara Women's University
Kitauoya Nishimachi, Nara 630-8506, Japan} \address{Department of mathematical
Sciences, University of Arkansas, Fayetteville, AR 72701}
\email{tsuyoshi@cc.nara-wu.ac.jp} \email{yoav@uark.edu} \author{Tsuyoshi
Kobayashi} \author{Yo'av Rieck} \thanks{The first names author was supported
by Grant-in-Aid for scientific research, JSPS grant number 19540083. The
second named author was supported in part JSPS (fellow number P00024) and by
the 21st century COE program ``Constitution for wide-angle mathematical basis
focused on knots" (Osaka City University); leader: Akio Kawauchi.}

\begin{document}

\subjclass{57M99, 57M25}%
\keywords{3-manifold, Heegaard splittings, strong irreducibility, weak reducibility}%

\date{\today}%

\begin{abstract}
We construct infinitely many manifolds admitting both strongly irreducible and
weakly reducible minimal genus Heegaard splittings.  Both closed manifolds and
manifolds with boundary tori are constructed.
\end{abstract}

\maketitle


The pioneering work of Casson and Gordon \cite{casson-gordon} shows that a
minimal genus Heegaard splitting of an irreducible, non-Haken 3-manifold is
necessarily strongly irreducible; by contrast, Haken \cite{haken} showed that
a minimal genus (indeed, any) Heegaard splitting of a composite 3-manifold is
necessarily reducible, and hence weakly reducible.  The following question of Moriah
\cite{moriah-top-appl-2004} is therefore quite natural:

\begin{question}[\cite{moriah-top-appl-2004}, Question~1.2]
Can a 3-manifold $M$ have both weakly reducible and strongly irreducible
minimal genus Heegaard splittings?
\end{question}

We answer this question affirmatively:

\begin{thm}
\label{thm:WRandSIHS}
There exist infinitely many closed, orientable 3-manifolds of Heegaard genus
3, each admitting both strongly irreducible and weakly reducible
minimal genus Heegaard splittings.
\end{thm}

Theorem~\ref{thm:WRandSIHS} is proved in Section~\ref{sec:proof}.  In Remark \ref{rmk:generalize} we offer a strategy to generalize Theorem \ref{thm:WRandSIHS} to construct examples of genus $g$, for each $g \geq 3$; it is easy to see that no such examples can exist if $g < 3$.  In Section \ref{sec:boundded} we give examples of manifolds with one, two or three torus boundary components, each admitting both strongly irreducible and weakly reducible minimal genus Heegaard splittings.  Moreover, for each manifold with two boundary components, we construct four minimal genus Heegaard surfaces, two weakly reducible, one separating the boundary components and one that does not, and similarly two strongly irreducible minimal genus Heegaard surfaces.  For a precise statement, see Theorem \ref{thm:boundded}.

In an effort to keep this article short we refer the reader to Section 2 of
\cite{KR-CAG} for definitions and background material.  Unless otherwise stated we
follow the notation of that paper.

\section{Preliminaries}
\label{sec:prelim}

\subsection{Constructing strongly irreducible Heegaard splittings}

In this subsection we introduce a method for constructing strongly irreducible Heegaard splittings
using 2-bridge link exteriors; this is taken out of \cite{kobayashi-polynomial-growth}.

\begin{dfns}
\begin{enumerate}
\item A {\it 2-string tangle} $(B^3;t_1,t_2)$ is a pair of 3-ball $B^3$ and two
disjoint arcs $t_1$, $t_2$ properly embedded in $B^3$.

\item A tangle is called {\it 2-string trivial tangle} if it is
homeomorphic (as a triple) to $(D^2 \times [0,1]; \{ p \} \times [0,1], \{ q \}
\times [0,1])$, where $D^2$ is a 2-disk and $p,q$ are two distinct
points in $\mbox{int}(D^2)$.
\end{enumerate}
\end{dfns}

Let $(B^3;t_1,t_2)$ be a 2-string trivial tangle. Let $H =
\mbox{cl}(B^3 \setminus (N(t_1) \cup N(t_2)))$, and $A_i =
\mbox{Fr}_{B^3}(N(t_i))$, $i=1,2$.  Note that $H$ is a genus 2
handlebody, $A_1$, $A_2$ are annuli in $\del H$ and the pair
$\{A_1,A_2\}$ is primitive in $H$ (see Figure~1), \ie, there
exist pairwise disjoint meridian disks $\Delta_1$, $\Delta_2
\subset H$ so that:

\begin{enumerate}
\item $\Delta_i \cap A_i$ is an essential arc in $A_i$ ($i=1,2$),
and
\item $\Delta_1 \cap A_{2}, \ \Delta_2 \cap A_1 = \emptyset$.
\end{enumerate}

\begin{figure}[ht]
\begin{center}
\includegraphics[width=6cm, clip]{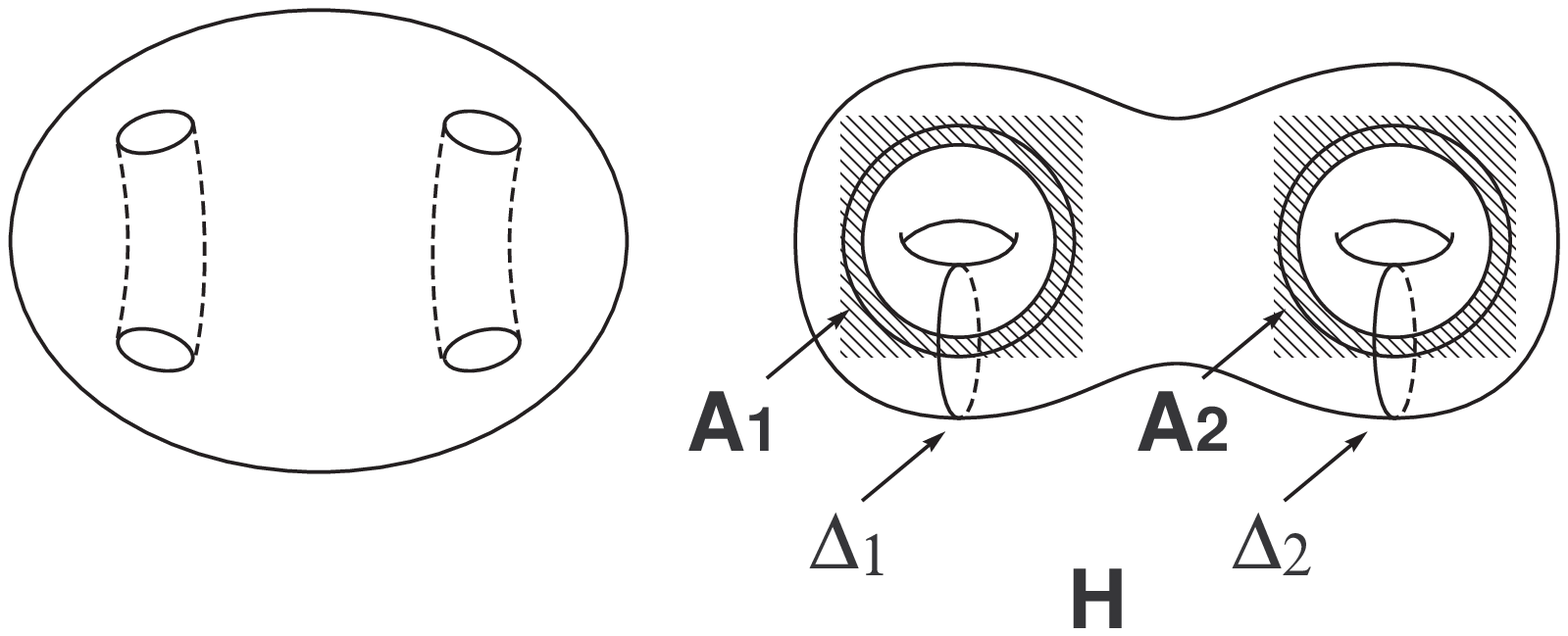}
\end{center}
\begin{center}
Figure 1.
\end{center}
\end{figure}

A link $L \subset S^3$ is called a {\it 2-bridge link} if it can be expressed
as the union of two 2-string trivial tangles; more precisely, if $(S^3; L) =
(B;t_1,t_2) \cup (B';t'_1,t'_2)$, where $(B;t_1,t_2)$ and $(B';t'_1,t'_2)$ are
2-string trivial tangles, $B \cap B' = \del B = \del B'$ and $L = (t_1 \cup
t'_1) \cup (t_2 \cup t'_2)$.  Note that in this paper by a 2-bridge link we
always mean a two component link, and not a 2-bridge knot.

Let $(H,A_1\cup A_2)$ be as above and  $(H',A'_1 \cup A'_2)$ be a copy of
$(H,A_1\cup A_2)$, $P = \mbox{cl}(\del H \setminus (A_1 \cup A_2))$ and
similarly  $P' = \mbox{cl}(\del H' \setminus (A'_1 \cup A'_2))$.  Let $L$
be a 2-bridge link.  Then we see from the above that there exists a homeomorphism $h:P \to P'$ such
that $E(L)$, the exterior of $L$, is homeomorphic to $H \cup_h H'$ and $\del E(L) = (A_1 \cup A'_1)
\cup (A_2 \cup A'_2)$, so that $\del A_i$, $\del A'_i$ are meridian curves
($i=1,2$).  The image of $P = P'$ in $E(L)$ is called a {\it bridge sphere}.

Let $N$ be a (possibly disconnected) orientable, irreducible,
$\del$-irreducible 3-manifold such that $\del N$ consists of two tori $T_1$
and $T_2$ and each component of $N$ has non empty boundary (hence, $N$
consists of at most two components).  Suppose that there exists a
3-dimensional submanifold $R \subset N$ such that:

\begin{enumerate}
\item Each component of $R$ is a handlebody, and $\mbox{Fr}_N(R)$ is
incompressible in $N$.

\item $T_i \cap R$ ($i=1,2$) consists of an annulus, say $\mathcal{A}_i$, such
that

    \begin{enumerate}
    \item $\mathcal{A}_i$ is incompressible in $N$, and---
    \item $\mathcal{A}_i$ is $\del$-incompressible in $R$ (\ie , there does
    not exist a disk properly embedded in $R$ that intersects $\mathcal{A}_i$
    in an essential arc).
    \end{enumerate}

\item Each component of $\mbox{cl}(N \setminus R) = R'$ is a handlebody such that
$T_i \cap R'$ ($i=1,2$) consists of an annulus, say $\mathcal{A}'_i$ satisfying:

    \begin{enumerate}
    \item $\mathcal{A}'_i$ is incompressible in $N$, and---
    \item $\mathcal{A}'_i$ is $\del$-incompressible in $R'$.
    \end{enumerate}

\end{enumerate}

With notation as above,
let $M$ be the 3-manifold obtained from $E(L)$ and $N$ by identifying their
boundary by an orientation reversing homeomorphism $\del N \to \del (E(L))$
such that $\mathcal{A}_i$ ($\mathcal{A}'_i$ resp.) is mapped to $A_i$ ($A'_i$
resp.). Let $V = H \cup R  \subset M$ and similarly $V' = H' \cup R' \subset
M$.  Since $A_1 \cup A_2$ ($A'_1 \cup A'_2$ resp.) is primitive in $H$ ($H'$
resp.), we see that $V$ ($V'$ resp.) is a handlebody obtained from $R$ ($R'$
resp.) by attaching a 1-handle (see Figure~2), and therefore $V \cup V'$ is a
Heegaard splitting of $M$.
\begin{figure}[ht]
\begin{center}
\includegraphics[width=8cm, clip]{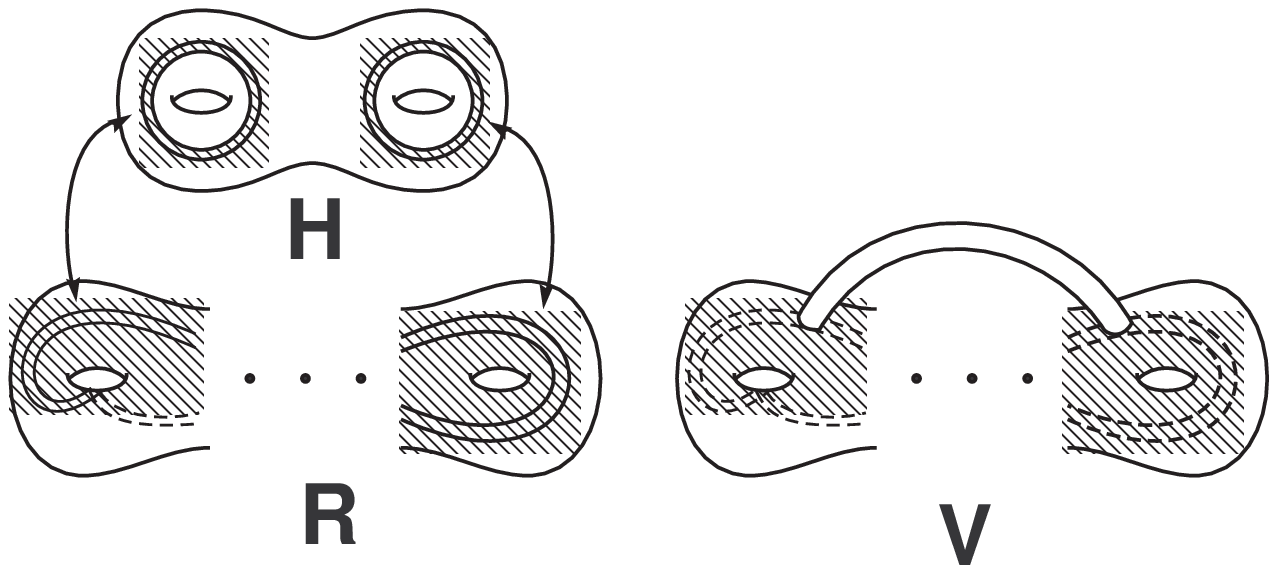}
\end{center}
\begin{center}
Figure 2.
\end{center}
\end{figure}
For this Heegaard splitting the following holds:

\begin{pro}
\label{pro:construct-SI}
With notation as above, if $L$ is not the trivial link or the Hopf link, then the Heegaard
splitting $V \cup V'$ is strongly irreducible.
\end{pro}

\begin{proof}[Sketch of proof]
Proposition \ref{pro:construct-SI} is identical to Proposition~3.1 of \cite{kobayashi-polynomial-growth} and the proof can be found there.  For the convenience of the reader, we sketch it here.  Let $D \subset V$ and $D' \subset V'$ be a pair of meridian disks.  Minimize the intersection of $D$ with $A_1 \cup A_2$, and the intersection of $D'$ with $A'_1 \cup A'_2$.  By symmetry, we have the following three cases:
\begin{enumerate}
\item $D \cap (A_1 \cup A_2) = \emptyset$ and $D' \cap (A'_1 \cup A'_2) = \emptyset$
\item $D \cap (A_1 \cup A_2) = \emptyset$ and $D' \cap (A'_1 \cup A'_2) \neq \emptyset$
\item $D \cap (A_1 \cup A_2) \neq \emptyset$ and $D' \cap (A'_1 \cup A'_2) \neq \emptyset$
\end{enumerate}
In the first case, $D$ (resp. $D'$) is the meridian disks of the tangles $(B;t_1,t_2)$ (resp. $(B';t'_1,t'_2)$); since $L$ is not the trivial link or the Hopf link, $D$ intersects $D'$ more than twice.  In the second case, $D$ is the meridian disk of the tangles $(B;t_1,t_2)$.  Consider an  outermost disk on $D'$, say $\delta'$. Note that $\delta' \subset H'$. If the arc of  $\delta'$ on $A_1'$ or $A_2'$ is inessential, we can pinch it off; the proof now is the same as the first case.  Else, $\delta'$ gives a boundary compression for $A_1'$ or $A_2'$.  Again, since $T$ is not the trivial link or the Hopf link, we see that $|D' \cap D| \geq |\delta' \cap D| >1$.

In the third case, we consider outermost disks, $\delta$ on $D$, and
$\delta'$ on $D'$. If the arc of $\delta$ on $A_1$ or $A_2$ is inessential,
or the arc of $\delta'$ on $A_1'$ or $A_2'$ is inessential, then arguments
similar to the above work.
Suppose $\delta$ on $A_1$ or $A_2$, and $\delta'$ on $A_1'$ or $A_2'$ are
essential. Since $L$ is not the trivial link or the Hopf link, we see that
$|D' \cap D| \geq |\delta' \cap \delta| \geq 1$.

\end{proof}

\subsection{Spines of amalgamated Heegaard splittings.}

A  {\it spine} of a compression body $C$ is a graph $\lambda$
embedded in $C$ so that $C \setminus (\lambda \cup \del_- C)$ is homeomorphic to $\del_+ C \times (-\infty,0]$.
Let $C \cup C'$ be a Heegaard splitting of a manifold $M$; a graph $\Gamma\subset M$ is a {\it spine for} $C$ if there exists an ambient isotopy of  $M$ so that the image of $\Gamma$ after this isotopy is contained in $C$ as a spine.
{\it Simultaneous spines of} $C \cup C'$ are two disjointly embedded graphs $\Gamma$, $\Gamma' \subset M$, so that after an ambient isotopy of $M$ the image of $\Gamma$ ($\Gamma'$ resp.) is contained in $C$ ($C'$ resp.) as a spine.  For the definition of amalgamation of Heegaard splitting see \cite{schultens}.

\begin{pro}
\label{pro:SpaineAfterAmalgamation}
Let $M_1$ and $M_2$ be manifolds so that $\del M_1$ and $\del M_2$ are connected and homeomorphic.  For $i=1,2$, let $H_i \cup C_i$ be Heegaard splittings of $M_i$, where $H_i$ is a handlebody and $C_i$ a compression body.  Let $\mu_i$ (resp. $\lambda_i$) be a spine of $H_i$ (resp. $C_i$).  Let $M$ be a manifold obtained by gluing $M_1$ and $M_2$ along their boundaries.  Let $H \cup H'$ be the amalgamation of $H_1 \cup C_1$ and $H_2 \cup C_2$.

Then there exist simultaneous spines of $H \cup H'$ so that $\mu_1 \cup \lambda_2$ is contained in a spine of $H$ or $H'$, and $\mu_2 \cup \lambda_1$ is contained in a spine of the other.
\end{pro}

\begin{proof}
We denote the image of $\del M_i$ in $M$ by $F$, the image of $\mu_i$ in $M$ by $\mu_i$, and the image of $\lambda_i$ in $M$ by
$\lambda_i$.
By transversality, we assume as we may that $\lambda_1 \cap \lambda_2 = \emptyset$.  The Heegaard surface that gives amalgamation of $H_1 \cup C_1$ and $H_2 \cup C_2$ is given by tubing $F$ along $\lambda_1$ into $M_1$ and along $\lambda_2$ into $M_2$, see Figure~3
(this figure is based on Schultens' \cite[Figure~3]{schultens}).
Note that the intersection of $F$ and the amalgamated Heegaard surface is {\it not}
transverse.

\begin{figure}[ht]
\begin{center}
\includegraphics[width=8cm, clip]{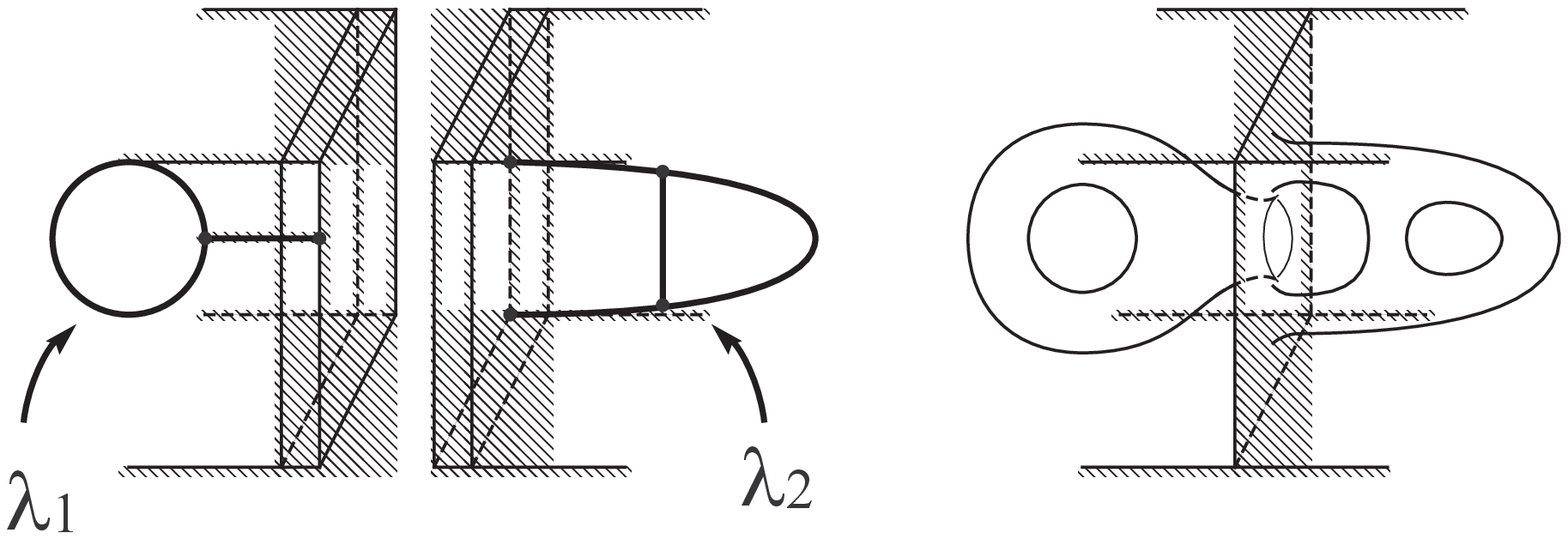}
\end{center}
\begin{center}
Figure 3.
\end{center}
\end{figure}

We may suppose that $\mu_1 \cup \lambda_2$ is contained in $H$ and $\mu_2 \cup \lambda_1$ is contained in $H'$.  By compressing $H$ along the disks $\mbox{cl}(\mbox{int}(H) \cap F)$ we obtain two handlebodies.  One handlebody is isotopic to $H_1$ and so we may take $\mu_1$ as its spine.  The other handlebody contains $\lambda_2$ and admits a
deformation retract onto it; moreover, $\lambda_2$ intersects each disk of $\mbox{cl}(\mbox{int}(H) \cap F)$ in exactly one point and has no other intersections with the boundary of this handlebody.  Since the two handlebodies were obtained from $H$ by compressing along
the disks $\mbox{cl}(\mbox{int}(H) \cap F)$, it is easy to construct a spine for $H$ by connecting $\lambda_2$ to $\mu_1$.
$H'$ is treated similarly; the proposition follows.
\end{proof}

\section{Proof of Theorem~\ref{thm:WRandSIHS}}
\label{sec:proof}

We adopt the notation of Section~\ref{sec:prelim}.

Let $3_1$ be the trefoil knot and $4_1$ the figure eight knot. Let $L = l_1
\cup l_2$ be a hyperbolic 2-bridge link.  Denote
$\del N(l_i)$ by $T_i$ ($i=1,2$).

We note that there exists an essential annulus $\bar{A}$ in $E(3_1)$ such that
the closures of the components of $E(3_1) \setminus \bar{A}$ are solid tori,
say $N_1$ and $N'_1$, where $\bar{A}$ wraps around $N_1$ longitudinally twice,
and around $N_1'$ longitudinally three times.  Hence, $N_1 \cap \del E(3_1)$ and
$N_1' \cap \del E(3_1)$ are incompressible and boundary incompressible.
On the other hand, we note that $4_1$ is a genus 1
fibered knot.  Hence we have the following: 
let $S \subset E(4_1)$ be a minimal genus Seifert surface for $4_1$ (note that $S$ is a once punctured torus). 
Let $N_2 = N(S)$ and $N'_2 = \mbox{cl}(E(4_1) \setminus N_2)$. Then
$N_2$ ($N'_2$ resp.) is homeomorphic to $S \times [0,1]$, where $N_2 \cap \del
E(4_1)$ ($N'_2 \cap \del E(4_1)$ resp.) corresponds to $\del S \times
[0,1]$. Note that $S \times [0,1]$ is homeomorphic to a genus 2 handlebody,
and $\del S \times [0,1]$ is incompressible and $\del$-incompressible in $S \times [0,1]$. Let $P$ be a bridge sphere in $E(L)$. 
Then as in Section~\ref{sec:prelim}, $P$ separates $E(L)$ into two genus 2
handlebodies, called $H$ and $H'$. Finally, let $M$ be a 3-manifold obtained from
$E(3_1) \cup E(4_1)$ and $E(L)$ by identifying their boundaries by a
homeomorphism $h:(\del E(3_1) \cup \del E(4_1)) \to \del E(L) (=T_1 \cup T_2)$
so that $h$ satisfies the following conditions:

\begin{enumerate}
\item $h(N_1 \cap \del E(3_1)) = H \cap T_1$, hence
$h(N'_1 \cap \del E(3_1)) = H' \cap T_1$.
\item $h(N_2 \cap \del E(4_1)) = H \cap T_2$, hence
$h(N'_2 \cap \del E(4_1)) = H' \cap T_2$.
\end{enumerate}

Note that the conditions of Proposition~\ref{pro:construct-SI} are satisfied,
and so we see that $M$ admits a strongly irreducible genus 3 Heegaard splitting.
Explicitly, the splitting surface is obtained from the bridge sphere
$P$ by attaching $\mbox{Fr}_{E(4_1)} N_2$ (that is, two once-punctured tori)
in $E(4_1)$ and $\bar{A}$ in $E(3_1)$.  Denote this splitting by $V \cup_\s
V'$, where $V$ and $V'$ are the handlebodies
$N_1 \cup H \cup N_2$ and $N_1' \cup H' \cup N_2'$ resp.,
and $\s$ is the splitting surface.

The decomposition $E(3_1) \cup E(L) \cup E(4_1)$ is the torus decomposition
for $M$. In \cite[Theorem]{kobayashi-genus-2-haken}, a complete list of
Heegaard genus 2 3-manifolds admitting non-trivial torus decomposition
is given. By
consulting that list, we see that $g(M) > 2$. Above we constructed a
strongly irreducible genus 3 Heegaard splitting for $M$.
We conclude that $g(M) = 3$, and that $M$ admits a
strongly irreducible minimal genus Heegaard splitting.

We claim that the submanifold $E(3_1) \cup E(L)$ admits a genus 2 Heegaard
splitting. Since $A_1$ is primitive in $H$ and $A_1'$ is primitive in $H'$,
$N_1 \cup H$ and $N_1' \cup H'$ are genus 2 handlebodies.
Let $A = H \cap T_2$ and $A' = H' \cap T_2$.  Let $C =
\mbox{cl}((N_1 \cup H) \setminus N(A,H))$ and $C' = (N_1' \cup H') \cup
N(A,H)$.  It is clear that $C$ is a genus 2 handlebody.
It is easy to see that
$A'$ is primitive in $N_1' \cup H'$, \ie, there is a meridian disk $\Delta'$
of $N_1' \cup H'$ such that $\Delta' \cap A'$ is an essential arc in
$A'$. This implies that $C'$ is a genus 2 compression body with $\partial_- C'
= A \cup A' = T_2$.  Denoting $\del_+ C$ by $\s'$, we see that $C \cup_{\s'} C'$
is a genus 2 Heegaard splitting of $E(3_1) \cup E(L)$.

\begin{rmkk}
\label{rmk:Genus2SplittingsLeft}
For future reference, we note the following: let $\alpha$ be a core curve of the
solid torus $N_1$ and $\alpha'$ a core curve of the solid torus $N_1'$.
By construction, $\alpha$ is contained in a spine of the handlebody $C$ and $\alpha'$ is contained in a spine of the compression body $C'$.  Similarly, the decomposition $M = \overline{C} \cup \overline{C}'$, where $\overline{C} = (N_1 \cup H) \cup N(A',H')$ and $\overline{C}' = \mbox{cl}((N_1' \cup H') \setminus N(A',H'))$, gives another (possibly isotopic) genus 2 Heegaard splitting of $E(3_1) \cup E(L)$ so that $\alpha'$ is contained in a spine of the handlebody $\overline{C}'$ and $\alpha$ is contained in a spine of the compression body $\overline{C}$.
\end{rmkk}

It is well known that $E(4_1)$ admits a genus 2 Heegaard splitting.
By amalgamating
a genus 2 Heegaard splitting for $E(4_1)$ with a genus 2 Heegaard splitting
of
$E(3_1) \cup E(L)$ we obtain a weakly reducible Heegaard splitting of $M$; by
\cite{schultens} (see also \cite[Lemma~2.7]{KR-CAG} for a more general
statement) this Heegaard splitting has genus 3.  This establishes the
existence of weakly reducible minimal genus Heegaard splittings of $M$.

This completes the proof of Theorem~\ref{thm:WRandSIHS}.

\begin{rmkk}
\label{rmk:generalize}
The following is a suggestion for a way to generalize the results of this paper.  Fix $g \geq 3$.  Let $H$ (resp. $H'$) be a genus $g-1$ handlebody, and $A_1$, $A_2 \subset \del H$ (resp. $A_1'$, $A_2' \subset \del H'$) be two primitive annuli. A construction similar to above gives handlebodies $V$, $V'$ of genus $g$.  The curve complex distance of a Heegaard splitting was defined by Hempel \cite{hempel-distance} and was generalized by several authors to bridge decompositions; note that $H \cup H'$ is a genus $g-3$, 2 bridge decomposition (for details see, for example, the proof of Proposition 2.2 of \cite{KR-AGT}).  It is reasonable to expect that if the distance of $H \cup H'$ is large, then $V \cup V'$ is strongly irreducible and minimal genus (Tomova's \cite{tomova} should be useful here).
Similar to the construction above, one obtains weakly reducible minimal genus Heegaard splittings by considering the decomposition $E(3_1) \cup H \cup H'$ and $E(4_1)$.  This would give manifolds of genus $g$, for arbitrary $g \geq 3$, admitting both weakly reducible and strongly irreducible minimal genus Heegaard splittings.
\end{rmkk}

\section{Further examples: the bounded case}
\label{sec:boundded}

Throughout this section, let $M = E(3_1) \cup E(L) \cup E(4_1)$ be any of the manifolds constructed in the previous section.  Let $V \cup_{\s} V'$ be the strongly irreducible Heegaard splitting constructed there.

Let $\beta^* \subset E(4_1)$ be the simple closed curve given in Figure~4.  
By Figure~4~(a), $\beta^*$ is contained in a once punctured torus that is a fiber of the fibration of $E(4_1)$ over $S^1$.  We may choose this fiber to be a component of $\s \cap E(4_1)$.

\begin{figure}[ht]
\begin{center}
\includegraphics[width=8cm, clip]{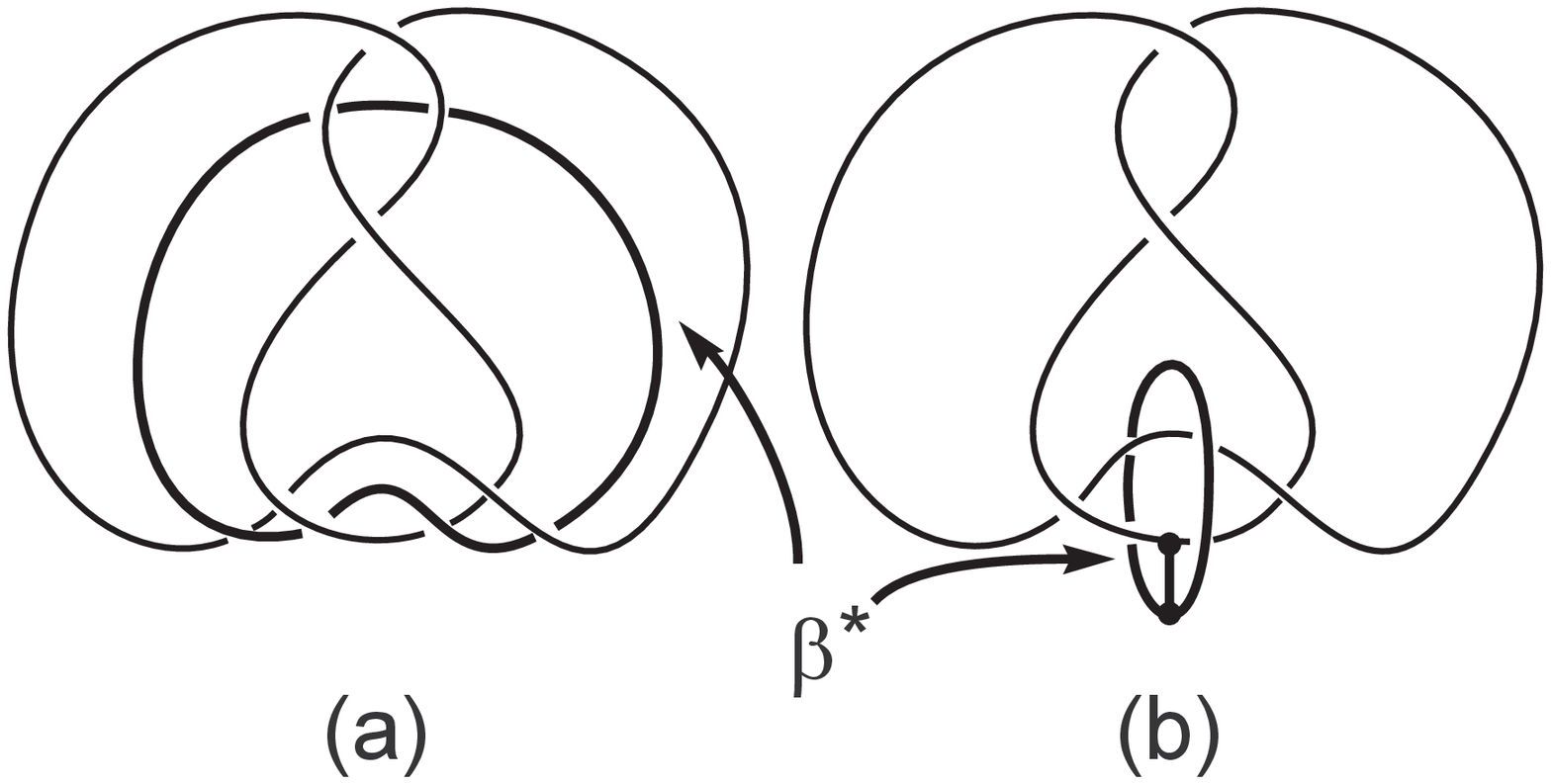}
\end{center}
\begin{center}
Figure 4.
\end{center}
\end{figure}

\begin{rmkk}
\label{rmk:beta*}
We connect $\beta^*$ to $\del E(4_1)$ by an arc as in Figure~4~(b).
It is directly observed that the exterior of a regular neighborhood
of ($\del E(4_1)$ together with the 1-complex) is a genus 2 handlebody.
This shows that
$\beta^*$ is contained in a spine of a compression body (not
handlebody) component of a genus 2 Heegaard splitting of $E(4_1)$.
\end{rmkk}

Let $\alpha$, $\alpha'$ be as in Remark~\ref{rmk:Genus2SplittingsLeft}, so that
$\alpha \subset V$ and $\alpha' \subset V'$.
Denote $\mbox{cl}(M \setminus N(\alpha \cup \beta^* \cup \alpha'))$ by $X$.
Denote the boundary components of $X$ by $T_\alpha = \del N(\alpha)$, $T_{\beta^*}
= \del N({\beta^*})$, and $T_{\alpha'} = \del N(\alpha')$.

\begin{lem}
\label{lem:WR}
$X$ admits two genus 3 weakly reducible Heegaard surfaces, denoted by $F_1$ and $F_2$, so that:
	\begin{enumerate}
	\item $F_1$ separates $T_\alpha \cup T_{\beta^*}$ and $T_{\alpha'}$.
	\item $F_2$ separates $T_{\alpha}$ and $T_{\alpha'} \cup T_{\beta^*}$.
	\end{enumerate}
\end{lem}

\begin{proof}
By applying Proposition~\ref{pro:SpaineAfterAmalgamation} to the Heegaard splitting
$C \cup C'$ (recall Remark~\ref{rmk:Genus2SplittingsLeft}) and the genus 2 Heegaard
splitting of $E(4_1)$ given in Remark~\ref{rmk:beta*} we obtain
a genus 3 Heegaard splitting of $M$ such that the Heegaard surface
separates $\alpha \cup \beta^*$ and $\alpha'$,
$\alpha \cup \beta^*$ is contained in a spine of one of the handlebodies,
and
$\alpha'$ is contained in a spine of the other handlebody.
This gives $F_1$.

Analogously, by applying Proposition~\ref{pro:SpaineAfterAmalgamation} to the Heegaard splitting
$\overline{C} \cup \overline{C}'$ (recall Remark~\ref{rmk:Genus2SplittingsLeft}) and
the genus 2 Heegaard
splitting of $E(4_1)$ given in Remark~\ref{rmk:beta*} we obtain $F_2$.
\end{proof}

\begin{lem}
$g(X) = 3$.
\end{lem}

\begin{proof}
Since $M$ is obtained from $X$ by Dehn filling, we have that
$g(X) \geq g(M) = 3$.  On the other hand, $F_1$ is a genus 3 Heegaard
surface for $X$, showing that $g(X) \leq g(F_1) = 3$.
\end{proof}

\begin{dfn}
Let $C$ be a compression body, and $\alpha_1,\dots,\alpha_n \subset C$
simple closed curves.
We say that $\alpha_1,\dots,\alpha_n$ are {\it simultaneous cores} if the following two conditions hold:
	\begin{enumerate}
	\item There exist mutually disjoint annuli $A_1,\dots,
          A_n \subset C$ so that one component of $\del A_i$ is
          $\alpha_i$ and the other is on $\del_+ C$.
	\item There exist mutually disjoint meridian disks $D_1,\dots,
          D_n \subset C$ so that $\alpha_i$ intersects $D_i$ transversely
          in one point and for $i \neq j$, $\alpha_i \cap D_j = \emptyset$.
	\end{enumerate}
\end{dfn}

\begin{rmkk}
\label{rmk:SimCores}
It is easy to see that $\alpha_1,\dots,\alpha_n \subset C$ are simultaneous
cores if and only if $\mbox{cl}(C \setminus N(\cup_{i=1}^n \alpha_i))$
is a compression body.
\end{rmkk}

Recall that $\beta^* \subset \s \cap E(4_1)$.
Let $\beta$ (resp. $\beta'$) be a curve obtained by pushing $\beta^*$
slightly into $V$ (resp. $V'$).

\begin{lem}
\label{lem:SimCores}
The curves $\alpha$, $\beta \subset V$ and $\alpha'$, $\beta' \subset V'$ are
simultaneous cores.
\end{lem}

\begin{proof}
Recall the definition of the handlebody $H = V \cap E(L)$ given in
Section~\ref{sec:prelim}, and let $\Delta_1$, $\Delta_2$ be the meridian disks of
$H$ shown in Figure~1.
Let $\widehat{D}_{\alpha}$ be a meridian disk of the solid torus $N_1 = V \cap E(3_1)$ that intersects the annulus $\bar{A} = \Sigma \cap E(3_1)$ essentially.  By attaching two copies of $\Delta_1$ to $\widehat{D}_{\alpha}$ we obtain a meridian disk for $V$, denoted by $D_{\alpha}$, that intersects $\alpha$ once and is disjoint from $\beta$.

Recall that $V \cap E(4_1) (=N_2)$ is homeomorphic to $S \times [0,1]$, where $S$ is a once
punctured torus.  We may suppose that $\beta$
corresponds to a curve $\beta_{S} \times \{ 1/2 \}$, where $\beta_{S}$
is an essential curve on $S$.
Let $\widehat{D}_{\beta}$ be a vertical disk in $V \cap E(4_1)$ that
intersects $\beta$ once, that is,  $\widehat{D}_{\beta}$ corresponds to a disk of the
form $\gamma \times [0,1]$, where $\gamma$ is an arc properly embedded in $S$
that intersects $\beta_{S}$ transversely once.  By attaching two copies of
$\Delta_2$ to $\widehat{D}_{\beta}$ we obtain a meridian disk for $V$, denoted by
$D_{\beta}$, that intersects $\beta$ once and is disjoint from $\alpha$.

It is easy to see that  $D_{\alpha} \cap D_{\beta} = \emptyset$, and that
there exist a pair of disjoint annuli, say $A_{\alpha}$ and
$A_{\beta}$, so that one component of $\del A_{\alpha}$ is $\alpha$ and the
other is on $\del V$ and one component of $\del A_{\beta}$ is $\beta$ and the
other is on $\del V$.   Hence, $\alpha$ and $\beta$ are simultaneous cores.

The curves $\alpha'$ and $\beta'$ are treated similarly.
\end{proof}

\begin{thm}
\label{thm:boundded}
For $i=1,2,3$, there exists infinitely many manifolds $M_i$ so that $\del M_i$ consists of exactly $i$ tori, $g(M_i) = 3$,
and each $M_i$ admits both strongly irreducible and weakly reducible
minimal genus Heegaard splittings.

Moreover, each manifold $M_2$ admits four distinct minimal genus Heegaard
surfaces, denoted $F_{SI}^{1,1}$, $F_{WR}^{1,1}$,  $F_{SI}^{2,0}$, $F_{WR}^{2,0}$, so that the following four conditions hold.
\begin{enumerate}
\item The Heegaard splittings given by $F_{SI}^{1,1}$ and
$F_{SI}^{2,0}$ are strongly irreducible.
\item The Heegaard splittings given by
$F_{WR}^{1,1}$ and $F_{WR}^{2,0}$ are weakly reducible.
\item $F_{SI}^{1,1}$ and $F_{WR}^{1,1}$ separate the two boundary components of $M_2$.
\item $F_{SI}^{2,0}$ and $F_{WR}^{2,0}$ do not separate the boundary components of $M_2$.
\end{enumerate}
\end{thm}

Before proving Theorem~\ref{thm:boundded} we give the following definition.

\begin{dfn}
Let $Y_1$ and $Y_2$ be manifolds so that $Y_1$ is obtained from $Y_2$ by Dehn filling
(equivalently, $Y_2$ is obtained from $Y_1$ by removing an open regular neighborhood
of a link in it). Note that $Y_2 \subset Y_1$.

Let $\s_2 \subset Y_2$ be any Heegaard surface.  Then $\s_2$ is a Heegaard
surface of $Y_1$.  We say that $\s_2 \subset Y_1$ is
an {\it induced Heegaard surface} (or the {\it Heegaard surface induced by} $\s_2$).

Let $\s_1 \subset Y_1$ be a Heegaard
surface.  Suppose that $\s_1 \subset Y_2$ and that $\s_1$ is a Heegaard surface of $Y_2$.
We say that $\s_1 \subset Y_2$ is
an {\it induced Heegaard surface} (or the {\it Heegaard surface induced by} $\s_1$.)
\end{dfn}

The proof of the following lemma is easy and left to the reader:

\begin{lem}
\label{lem:induced}
Let $Y_1$ and $Y_2$ be as above.  If a Heegaard surface
$\s_2$ of $Y_2$ is weakly reducible,
then so is the induced Heegaard surface.  On the other hand, if $\s_1
\subset Y_1$ is a strongly irreducible Heegaard surface that induces a Heegaard surface
for $Y_2$, then the induced Heegaard surface is strongly irreducible.
\end{lem}

\begin{proof}[Proof of Theorem~\ref{thm:boundded}]
We deal with the cases $i = 1$, $i = 2$ and $i = 3$ in increasing order of difficulty.

\smallskip

\noindent For $i = 3$, let $M_3 = X$.  Then by Lemma \ref{lem:WR}, $g(X) = 3$ and
$X$ admit a weakly reducible minimal genus Heegaard splitting.

Note that $\beta^*$
is isotopic to $\beta$; hence $X$ is homeomorphic to
$\mbox{cl}(M \setminus N(\alpha \cup \alpha' \cup \beta))$.
By Lemma \ref{lem:SimCores} and Remark \ref{rmk:SimCores}, $V \cup V'$ induces
a genus 3 Heegaard splitting of
$\mbox{cl}(M \setminus N(\alpha \cup \alpha' \cup \beta))$.  Since
$V \cup V'$ is strongly irreducible, Lemma~\ref{lem:induced} shows that
the induced Heegaard splitting is  strongly irreducible.  The
case $i = 3$ follows.

\bigskip

\noindent For $i=1$, let $M_1 = \mbox{cl}(M \setminus N(\alpha))$.
Then $g(M_1) \geq g(M) = 3$.  Since $X$ is obtained from $M_1$ by removing an
open neighborhood of $\alpha'$ and $\beta^*$, $g(M_1) \leq g(X) = 3$.  We see that
$g(M_1) = 3$.

Note that $M_1$ is obtained by filling two boundary components of $X$.  Hence the genus 3 weakly reducible Heegaard splittings for $X$ given in Lemma \ref{lem:WR} induces genus 3 weakly reducible Heegaard splittings for $M_1$.

By Lemma \ref{lem:SimCores} and Remark \ref{rmk:SimCores}, $V \cup V'$ induces a genus 3
Heegaard splitting for $M_1$.  As above, the induced Heegaard splitting is strongly
irreducible.  The case $i=1$ follows.

\bigskip

\noindent For $i=2$, let $M_2 = \mbox{cl}(M \setminus N(\alpha \cup \beta^*))$.
Similar to $M_1$, it is easy to see that $g(M_2) = 3$.

By Lemma \ref{lem:induced}, each of the two genus 3 weakly reducible Heegaard splittings given in Lemma~\ref{lem:WR} induces a genus 3 weakly reducible Heegaard splitting on $M_2$, one not separating the components of $\del M_2$ (corresponding to Lemma \ref{lem:WR} (1)),
and the other separating them (corresponding to Lemma \ref{lem:WR} (2)).  These are the surfaces $F_{WR}^{2,0}$ and $F_{WR}^{1,1}$ in the theorem.

Note that $\beta$ ($\beta'$ resp.) is isotopic to $\beta^*$; hence, $M_2$ is homeomorphic to
$\mbox{cl}(M \setminus N(\alpha \cup \beta))$
($\mbox{cl}(M \setminus N(\alpha \cup \beta'))$ resp.).
By Lemma~\ref{lem:SimCores}, $V \cup V'$ induces a Heegaard splitting for
$\mbox{cl}(M \setminus N(\alpha \cup \beta))$
that does not separate the boundary components of
$\mbox{cl}(M \setminus N(\alpha \cup \beta))$.  The corresponding Heegaard surface for $M_2$ is the surface $F_{SI}^{2,0}$.
Similarly, by Lemma~\ref{lem:SimCores}, $V \cup V'$ induces a Heegaard splitting for
$\mbox{cl}(M \setminus N(\alpha \cup \beta'))$
that separates the boundary components of
$\mbox{cl}(M \setminus N(\alpha \cup \beta'))$.  The corresponding Heegaard surface for $M_2$ is the surface $F_{SI}^{1,1}$.
By Lemma~\ref{lem:induced}, $F_{SI}^{2,0}$ and $F_{SI}^{1,1}$
are strongly irreducible.  The case $i=2$ follows.
\end{proof}


\providecommand{\bysame}{\leavevmode\hbox to3em{\hrulefill}\thinspace}
\providecommand{\MR}{\relax\ifhmode\unskip\space\fi MR }
\providecommand{\MRhref}[2]{%
  \href{http://www.ams.org/mathscinet-getitem?mr=#1}{#2}
}
\providecommand{\href}[2]{#2}

\end{document}